\documentclass[12pt]{amsart}
\usepackage{amssymb,latexsym, amsmath, amscd, array, graphicx}

\swapnumbers \numberwithin{equation}{section}

\theoremstyle{plain}

\newtheorem{thm}{Theorem}[section]

\newtheorem{prop}[thm]{Proposition}
\newtheorem{cor}[thm]{Corollary}

\theoremstyle{definition}
\newtheorem{defin}[thm]{Definition}

\newtheorem{question}[thm]{Question}

 \newcommand{\Wi}{\widetilde}

\DeclareMathOperator{\as}{{\rm asdim}}

\def\Z{{\mathbb Z}}

\def\1{\hbox{\rm\rlap {1}\hskip.03in{\rom I}}}
\def\Bbbone{{\rm1\mathchoice{\kern-0.25em}{\kern-0.25em}
{\kern-0.2em}{\kern-0.2em}I}}

\long\def\forget#1\forgotten{} %

\newcommand\ver[1]{\marginpar{\tiny Changed in Ver \VER}}

\date{\today}
\begin{document}

\title[macroscopic dimension]{On macroscopic dimension of universal coverings of closed manifolds}

\author[A.~Dranishnikov]{Alexander  Dranishnikov} %

\address{Alexander N. Dranishnikov, Department of Mathematics, University
of Florida, 358 Little Hall, Gainesville, FL 32611-8105, USA \\
and Steklov Mathematical Institute, Moscow, Russia}
\email{dranish@math.ufl.edu}

\subjclass[2000]
{Primary 55M30; 
Secondary 53C23,  
57N65  
}

\keywords{}

\begin{abstract}
We give a homological characterization of $n$-manifolds whose universal covering $\Wi M$ has Gromov's macroscopic dimension
$\dim_{mc}\Wi M<n$. As the result we distinguish $\dim_{mc}$ from the macroscopic dimension $\dim_{MC}$ defined by the author~\cite{Dr}. 
We prove the inequality $\dim_{mc}\Wi M<\dim_{MC}\Wi M=n$
for every closed $n$-manifold $M$ whose fundamental group $\pi$ is a geometrically finite amenable duality group
with  the cohomological dimension $cd(\pi)> n$.
\end{abstract}

\maketitle \tableofcontents

\section {Introduction}

Gromov introduced the macroscopic dimension $\dim_{mc}$ to study manifolds with positive scalar curvature~\cite{Gr1}.
He noticed that the universal covering $\Wi M$ of a Riemannian $n$-manifold with positive scalar curvature is dimensionally deficient on large scales.
Gromov formulated his observation as the conjecture: {\em For the universal covering $\Wi M$ of an $n$-manifold with positive scalar curvature}
$$\dim_{mc}\Wi M<n-1. $$  
The macroscopic dimension $\dim_{mc}$ was defined as follows
\begin{defin}[\cite{Gr1}] For a Riemannian manifold $X$ its macroscopic dimension does not exceed $n$,  $\dim_{mc}X\le n$, if there is a continuous map $g:X\to K^n$ to an $n$-dimensional simplicial complex with a uniform upper bound on the size of the preimages: $diam (g^{-1}(y))<b$ for all $y\in K^n$.

Gromov called such maps {\em uniformly cobounded}.
\end{defin}
D. Bolotov proved this conjecture for 3-manifolds~\cite{Bol1}.
In~\cite{BD} we proved Gromov's conjecture for spin manifolds with the fundamental groups $\pi$ satisfying the Analytic Novikov conjecture and the injectivity condition of the natural transformation $ko_*(B\pi)\to KO(B\pi)$ of the real connective 
$K$-theory to the periodic one.

The definition of $\dim_{mc} X$ is good for any metric space $X$. This concept differs drastically from the notion of the asymptotic dimension $\as X$
also defined by Gromov for all metric spaces. The former is an instrument to study open Riemannian manifolds and the later is the best
for finitely generated groups taken with the word metric. We note that there is the inequality $\dim_{mc}X\le\as X$ for all metric spaces.

We call an $n$-manifold $M$ {\em md-small} if for its universal cover $\Wi M$ we have the inequality $\dim_{mc}\Wi M< n$. The Gromov conjecture states in particular that
 a closed Riemannian manifold with  positive scalar curvature is $md$-small.
The main source of $md$-small manifolds are inessential manifolds. In fact the $md$-smallness of manifolds with positive scalar curvature in \cite{Bol1} and \cite{BD} was obtained as a consequence of their inessentiality.

We recall that an $n$-manifold $M$
is {\em inessential} if a map $f:M\to B\pi$ that classifies its universal covering can be deformed to the
$(n-1)$-skeleton $B\pi^{(n-1)}$. It is known that an orientable manifold is inessential if and only if $f_*([M])=0$
for the integral homology where $[M]$ is the fundamental class. An orientable manifold $M$ is called
{\em rationally inessential} if $f_*([M])= 0$ for the rational homology. Gromov asked~\cite{Gr1} whether
every $md$-small manifold is rationally inessential. It turns out that the answer is negative~\cite{Dr}.
In order to solve this problem I modified Gromov's definition of macroscopic dimension by introducing  in~\cite{Dr} a more geometric invariant  denoted as $\dim_{MC}$.

\begin{defin}[\cite{Dr}]\label{def2} For a metric space $\dim_{MC}X\le n$ if there is a uniformly cobounded Lipschitz map 
$g:X\to K\subset\ell_2$ to an $n$-dimensional simplicial complex realized in the standard  simplex of a Hilbert space.
\end{defin}
We note that with the same result $K$ can be equipped with the geodesic metric induced by the imbedding $K\subset\ell_2$.

We recall that the asymptotic dimension is defined as follows~\cite{Gr2},\cite{BeD}: {\em $\as X\le n$ if for every $\epsilon>0$ there  is a uniformly cobounded $\epsilon$-Lipschitz map 
$g:X\to K\subset\ell_2$ to an $n$-dimensional simplicial complex realized in the standard  simplex of a Hilbert space.}

It  follows  immediately from the definitions that
$$
(*)\ \ \ \ \ \ \ \ \ \ \ \ \ \dim_{mc}X\le\dim_{MC}X\le\as X.
$$
In the case when $X=\Wi M$ is the universal cover of a closed manifold $M$ with the lifted from $M$ metric it is easy to see that all three dimensions do not depend on the choice of the metric on $M$. In that case
the second inequality can be strict as it can be seen from the following example: If $M$ is a closed smooth 4-manifold
with the fundamental group $\pi_1(M)=\mathbb Z^5$, then $\dim_{MC}\Wi M\le 4$ and $\as\Wi M=\as\mathbb Z^5=5$.

The first inequality in $(*)$ is a different story. It is not easy to come with an example of a metric space $X$ that distinguishes these two dimensions. Moreover, in ~\cite{Dr2} I have excluded the Lipschitz condition from the definition of  $\dim_{MC}$ assuming that it is superfluous and it always can be achieved for reasonable metric spaces like Riemannian manifolds. 
Recently I discovered to my  surprise that these  two concepts are  different even for the universal coverings of closed manifolds. 

In this paper we present an example of an  $n$-manifold $M$ with $$\dim_{mc}\Wi M< \dim_{MC}\Wi M=n.$$
In fact we show that this conditions hold true for every closed $n$-manifold $M$ whose fundamental group $\pi$ is a geometrically finite amenable duality group with the cohomological dimension $cd(\pi)>n$.
This result (and the example) became possible after developing some (co)homological theory of $md$-small manifolds parallel to
that for $\dim_{MC}$ from~\cite{Dr},\cite{Dr2}. It was done in \S 3-5. In \S 2 we present some general geometric and topological properties of $md$-small manifolds. 
\subsection{Acknowledgment} The author would like to thank FIM-ETH Zurich for the hospitality. Also he is thankful to the referee for valuable remarks.

\section{ $md$-Small manifolds}

We recall that a cover $\mathcal U=\{U_{\alpha}\}_{\alpha\in J}$ of a metric space $X$ is called open if all $U_{\alpha}$ are open, it is uniformly bounded if there  is $a>0$ such that $diam U_{\alpha}<a$ for all $\alpha\in J$. It has the order $\le n$ if every point of $X$ is covered by no more than $n$ elements of $\mathcal U$. A cover $\mathcal U$ of $X$ is called irreducible if no subfamily of $\mathcal U$
covers $X$. A cover $\mathcal W$ of $X$ refines a cover $\mathcal U$ if for every $W\in\mathcal W$ there is $U\in\mathcal U$
such that $W\subset U$. We recall that $\epsilon>0$ is a Lebesgue number for a cover $\mathcal U$ of $X$ if for every subset $A\subset X$ of diameter $<\epsilon$ there is $U\in\mathcal U$ such that $A\subset U$.

A metric space is called {\em proper} if every closed ball with respect to that metric is compact.

\begin{prop}\label{md-charact}
For a proper  metric space $X$ the following are equivalent:

(1) $\dim_{mc}X\le n$;

(2) There is a uniformly bounded irreducible open cover $\mathcal U$ of $X$ with the order $\le n+1$.

(3) There is a continuous map $f:X\to K$ to a locally finite $n$-dimensional simplicial complex and a number $b>0$ such that 
$f^{-1}(\Delta)\ne\emptyset$ and $diam f^{-1}(\Delta)< b$ for every simplex $\Delta\subset K$.
\end{prop}
\begin{proof}
$(1) \ \Rightarrow \ (2)$. 
Let $g:X\to N$  be a continuous map to an $n$-dimensional simplicial complex and $a>0$ be such that
$diam g^{-1}(y)< a$.  We may assume that $N$ is countable.  For every $y\in N$ there is a neighborhood $V_y$ such that $diam g^{-1}(V_y)<a$. Consider a cover $\mathcal V$ of $N$.
Since $\dim N\le n$ there is an open locally finite cover $\mathcal W$ of $N$ that refines $V$ and has the order $\le n+1$.
It is easy to prove using the Zorn Lemma that every locally finite cover has an irreducible subcover.
We define $\mathcal U$ to be an irreducible subcover of $g^{-1}(\mathcal W)=\{g^{-1}(W)\mid W\in\mathcal W\}$.

$(2) \ \Rightarrow \ (3)$. Let $\mathcal U$ be a uniformly bounded irreducible open cover of $X$ of order $\le n+1$.

Claim: Every  closed ball $B(x,r)$ in $X$ is covered by finitely many elements of $\mathcal U$. 

Assume the contrary. Consider the subfamily $\mathcal A=\{U\in\mathcal U\mid U\cap B(x,r+a)\ne\emptyset\}$ where $a$ is an upper bound on size of $U$. Since the ball $B(x,r+a)$ is compact, there is a finite subcover $\mathcal B\subset\mathcal A$ of $B(x,r+a)$.
Let  $\mathcal C=\{U\in\mathcal U\mid U\cap (X\setminus B(x,r+a))\ne\emptyset\}$. Then every $U\in\mathcal  C$ does not intersect
$B(r,a)$. Therefore, $\mathcal V=\mathcal B\cup\mathcal C$ is a proper subcover of $\mathcal U$ which contradicts to the irreducibility condition.

Using a partition of unity subordinated by $\mathcal U$ we construct a projection $g:X\to K$ of $X$ to the nerve
$N(\mathcal U)=K$ of $\mathcal U$. The irreducibility of $\mathcal U$ implies that $f^{-1}(v)\ne\emptyset$ for every vertex $v\in K$.
Since the order of $\mathcal U$ does not exceed $n+1$, we obtain $\dim K\le n$. The above claim implies that every 
$U\in\mathcal U$ intersects only finitely many other elements of $\mathcal U$. Hence every vertex in $K$
has bounded valency. Thus, $K$ is locally finite.

Clearly, $(3) \ \Rightarrow \ (1)$.
\end{proof}

We recall that a real-valued function $f:X\to\mathbb R$ is called upper semicontinuous if for every $x\in X$ and $\epsilon>0$,
there is a neighborhood $U$ of $x$ such that $f(y)\le f(x)+\epsilon$ for all $y\in U$. Note that for a proper continuous map 
$g:X\to Y$ of a metric space the function $\psi:Y\to\mathbb R$ defined as $\psi(y)=diam g^{-1}(y)$ is upper semicontinuous.

A homotopy $H:X\times I\to Y$  in a metric space $Y$ is called {\em bounded} if there is a uniform upper bound on the diameter
of paths $H(x\times I)$, $x\in X$.

For a finitely presented group $\pi$ we can choose an Eilenberg-McLane complex $K(\pi,1)=B\pi$ to be locally finite and hence metrizable. We consider a proper geodesic metric on $B\pi$ and lift it to the universal cover $E\pi$. Thus, every closed ball in $B\pi$ is compact.

We recall that a metric space $Y$ is {\em uniformly contractible} if for every $R>0$ there is $S>0$ such that every $R$-ball $B(y,R)$ in $Y$
can be contracted to a point in the ball $B(y,S)$. Note that in the case of finite complex $B\pi$ the space $E\pi$ is uniformly contractible.
Generally, $E\pi$ is uniformly contractible over any compact set $C\subset B\pi$: {\em For every $R>0$ there is $S>0$ such that every $R$-ball $B(y,R)$ in $Y$
with $p(y)\in C$ can be contracted to a point in the ball $B(y,S)$ where $p:E\pi\to B\pi$ is the universal covering. }

\begin{thm}\label{st-cech} 
For a closed $n$-manifold $M$ the following conditions are equivalent:

(1) $M$ is $md$-small.

(2) A lift $\Wi f:\Wi M\to E\pi$ of the classifying map $f:M\to B\pi$ to the universal coverings 
can be deformed to a map $q:\Wi M\to E\pi^{(n-1)}$  by a bounded homotopy $H:\Wi M\times[0,1]\to E\pi$ for any choice of proper geodesic metric on $B\pi$.

The homotopy $H$ can be taken to be cellular for any choice of CW structure on $M$ and the product CW structure on $\Wi M\times[0,1]$.

\end{thm}
\begin{proof}
$(1)\ \Rightarrow\ (2)$. Let $g:\Wi M\to K$ be a continuous map to an $(n-1)$-dimensional locally finite simplicial complex
from Proposition~\ref{md-charact}. We recall that $g^{-1}(v)\ne\emptyset $ for all vertices $v\in K$ and $diam g^{-1}(\Delta)<b$ for every simplex $\Delta$. Let $\Wi f:\Wi M\to E\pi$ be an induced by $f:M\to B\pi$ map
of the universal covers. We may assume that $f$ is Lipschitz. Then so is $\Wi f$. Using the uniform contractibility of $E\pi$ over compact sets by induction on the dimension of the skeleton of $K$
we construct a map $h:K\to E\pi$ as follows.   Since $g^{-1}(v)\ne\emptyset$ for each vertex $v\in K$, we can fix $h(v)\in \Wi f(g^{-1}(v))$.
For every edge $[v,v']\subset K$ we have $$d(h(v),h(v'))=d(\Wi f(u),\Wi f(u'))\le\lambda d_{\Wi M}(u,u')\le \lambda b$$
where $u\in g^{-1}(v)$, $u'\in g^{-1}(v')$, and $\lambda$ is a fixed Lipschitz constant for $\Wi f$.
Since $E\pi$ is uniformly contractible over $f(M)$ there is $d_1>0$ the same for all edges such that
we can join $h(v)$ and $h(v')$ by a path of diameter $\le d_1$. Since the covering map $p:E\pi\to B\pi$ is 1-Lipschitz, the image $p(h(K^{(1)}))$ lies in the $d_1$-neighborhood of $f(M)$. In view of the properness of the metric on $B\pi$ the set $p(h(K^{(1)}))\subset B\pi$ is compact. Hence we can use the uniform contractibility of $E\pi$ over $p(h(K^{(1)})$ to define $h$ on $K^{(2)}$ and so on.
As the result we construct a map $h:K\to E\pi$ such that $h\circ g$ is in a finite distance from  $\Wi f$.

Then we take a cellular approximation $q:\Wi M\to E\pi^{(n-1)}$ of $h\circ g$. Again, the the map $q:\Wi M\to E\pi$ is in a finite distance from $\Wi f$. Therefore it can be joined with $\Wi f$ by a bounded homotopy $H:\Wi M\times[0,1]\to E\pi$. Then we take a cellular approximation of $H$ rel $M\times\{0,1\}$.

$(2)\ \Rightarrow\ (1)$. Without loss of generality we may assume that $B\pi$ and hence $E\pi$ are simplicial complexes.
Let a map $g:\Wi M\to E\pi^{(n-1)}$ be at distance $< a$ from $\Wi f$. We show that $g$ is uniformly cobounded.
Let $N=N_a(f(M))$ be the closed $a$-neighborhood of $f(M)$ in $B\pi$.  We consider a function
$\psi:N\to\mathbb R$ defined as $\psi(z)=diam\Wi f^{-1}(B(y,a))$ where $y\in p^{-1}(z)$.
Clearly $\psi$ is well-defined. Since $\psi$ is upper semicontinuous, the compactness of $N$ implies that it is bounded from above.

Thus,
there is $b>0$ such that $diam \Wi f^{-1}(B(x,a))<b$ for all $x\in N$. If $x_1,x_2\in g^{-1}(y)$, then $\Wi f(x_1),\Wi f(x_2)\in B(y, a)$. Therefore, $d_{\Wi M}(x_1,x_2) <b$.
\end{proof}
REMARK. In Theorem~\ref{st-cech}  $M$ can be taken to be any finite complex.

\

We recall first that a map between metric spaces $f:X\to Y$  is
called {\em coarsely Lipschitz} if there are positive constants
$(\lambda,b)$ such that
$$d_Y(f(x),f(x'))\le\lambda d_X(x,x')+b$$
for all $x,x'\in X$.
\begin{cor}\label{2.4}
A closed $n$-manifold $M$ is md-small if and only if for any choice of locally finite CW complex $B\pi$ and a proper geodesic metric on it
the universal cover $\Wi M$ admits a uniformly cobounded coarsely Lipschitz cellular map
$g:\Wi M \to E\pi^{(n-1)}$ into the $(n-1)$-skeleton of $E\pi$ with the compact image $pg(\Wi M)$ where the metric on $E\pi$ 
is lifted from $B\pi$ and $p:E\pi\to B\pi$ is universal covering map.
\end{cor}
\begin{proof} The backward implication is trivial.
For the forward implication we consider a map $g:\Wi M \to E\pi^{(n-1)}$ from  Theorem~\ref{st-cech}(2). 
In the proof of the implication $(2)\Rightarrow (1)$ there it was shown that $g$ is uniformly cobounded.
It  is coarsely Lipschitz, since it is in finite distance from a Lipschitz map $\Wi f$. The last condition of the corollary is satisfied by the same reason 
and the fact that $B\pi$ is locally compact.
\end{proof}
Similarly $n$-manifolds $M$ with $\dim_{MC}\Wi M<n$ will be called {\em MD-small}. A similar result holds for MD-small manifolds.
\begin{thm}[\cite{Dr}]\label{st-C} 
For a closed $n$-manifold $M$ the following conditions are equivalent:

(1) $M$ is MD-small.

(2) A lift $\Wi f:\Wi M\to E\pi$ of the classifying map $f:M\to B\pi$ to the universal coverings 
can be deformed to the $(n-1)$-dimensional skeleton $E\pi^{(n-1)}$ of $E\pi$ by a Lipschitz  homotopy where the metric on $E\pi$
is lifted from the metric on $B\pi$ for any choice of a locally finite complex $B\pi$ and a proper geodesic metric on it.
\end{thm}

The dimension $\dim_{MC}$ can be defined in terms of open covers as follows.
\begin{prop}\label{MD-charact}
For a proper  metric space $X$ the following are equivalent:

(1) $\dim_{MC}X\le n$;

(2) There is a uniformly bounded irreducible open cover $\mathcal U$ of $X$ of the order $\le n+1$
and with a Lebesque number $>0$.
\end{prop}
Thus the concepts $\dim_{mc}$ and $\dim_{MC}$ stated in the language of covers deffer only by the Lebesgue number condition.

\section{Coarsely equivariant cohomology}

Let $X$ be a CW complex and let $E_n(X)$ denote the set of its
$n$-dimensional cells. We recall that (co)homology of a CW complex
$X$ with coefficients in an abelian group $G$ are defined by 
means of the cellular chain complex $C_*(X)=\{C_n(X),\partial_n\}$
where $C_n(X)$ is the free abelian group generated by the set
$E_n(X)$. The resulting groups $H_*(X;G)$ and $H^*(X;G)$ do not
depend on the choice of the CW structure on $X$. The proof of this
fact appeals to the singular (co)homology theory and it is a part of
all textbooks on algebraic topology. The same holds true for
(co)homology groups with locally finite coefficients, i.e., for
coefficients in a $\pi$-module $L$ where $\pi=\pi_1(X)$. The chain
complex defining the homology groups $H_*(X;L)$ is $\{C_n(\Wi
X)\otimes_{\pi}L\}$ and the cochain complex defining the cohomology
$H^*(X;L)$ is $Hom_{\pi}(C_n(\Wi X),L)$ where $\Wi X$ is the
universal cover of $X$ with the cellular structure induced from $X$.
The resulting groups $H_*(X;L)$ and $H^*(X;L)$ do not depend on the
CW structure on $X$.

These groups can be interpreted as the equivariant (co)homology:
$$H_*(X;L)=H_*^{lf,\pi}(\Wi X;L)\ \ \ and\ \ \ H^*(X;L)=H_{\pi}^*(\Wi X;L).$$
The last equality is obvious since the equivariant cohomology groups
$H_{\pi}^*(\Wi X;L)$ are defined by equivariant cochains
$C_{\pi}^n(\Wi X,L)$, i.e., homomorphisms $\phi:C_n(\Wi X)\to L$ such that
the set $$S_{\phi,c}=\{\gamma^{-1}\phi(\gamma c)\mid \gamma\in\pi\}$$ consists of one element  for
every $c\in C_n(\Wi X)$. In~\cite{Dr} we defined the group of the almost equivariant cochains
$C_{ae}^n(\Wi X,L)$ that consists of homomorphism $\phi:C_n(\Wi X)\to L$ for which the set
$S_{\phi,c}$ is finite for every chain $c\in C_n(\Wi X)$.
Here we consider the group of coarsely equivariant cochains $C_{ce}^n(\Wi X,L)$ that consists of
homomorphisms $\phi: C_n(\Wi X)\to L$ such that the set $S_{\phi,c}\subset L$ 
generates a finitely generated abelian group for every $c\in C_n(\Wi X)$.
Thus we have a string of the inclusions of cochain complexes
$$
C^*_{\pi}(\Wi X,L)\subset C^*_{ae}(\Wi X,L)\subset C^*_{ce}(\Wi X,L)\subset C^*(\Wi X,L)
$$
which induces a chain of homomorphisms of corresponding cohomology groups
$$
H^n(X;L)\stackrel{\beta}\longrightarrow H^n_{ae}(\Wi X;L)\stackrel{\alpha}\longrightarrow H^n_{ce}(\Wi X;L)\to H^n(\Wi X;L).
$$
The first homomorphism was denoted in~\cite{Dr} as $\beta=pert^*_{X}$.
We call $ec^*_X=\alpha\circ pert^*_X$ the {\em equivariant coarsening} homomorphism.
Note that  the last arrow is an isomorphism for  $L$ finitely generated as abelian groups.
The cohomology groups $H^*_{ae}(\Wi X;L)$ are called the {\em almost equivariant
cohomology} of $\Wi X$ with coefficients in a $\pi$-module $L$~\cite{Dr} and the cohomology
groups $H^*_{ce}(\Wi X;L)$ are called the {\em coarsely equivariant
cohomology} of $\Wi X$ with coefficients in a $\pi$-module $L$.

One can define singular coarsely equivariant cohomology by replacing
cellular $n$-chains $c$ in the above definition by  singular $n$-chains.
 The standard argument show that the
singular version of coarsely equivariant cohomology coincides with
the cellular. Thus the group $H^*_{ce}(\Wi X;L)$ does not depend on
the choice of a CW complex structure on $X$.
\\
\\

REMARK. In the case when $L=\Z$ is a trivial module the cohomology
theory $H^*_{ae}(\Wi X;\Z)$ coincides with the
$\ell_{\infty}$-cohomology, in particular, when $X=E\pi$ it is the
$\ell_{\infty}$-cohomology of a group $H_{(\infty)}^*(\pi;\Z)$ in
the sense of~\cite{Ger}. Whereas the coarsely equivariant cohomology
$H^*_{ce}(\Wi X;\Z)$ equals the standard (cellular) cohomology of
the universal cover $H^*(\Wi X;\Z)$.
\\
\\

A  proper cellular map $f:X\to Y$ that induces an
isomorphism of the fundamental groups lifts to a proper cellular map
of the universal covering spaces $\bar f:\Wi X\to\Wi Y$. The lifting
$\bar f$ defines a chain homomorphism $\bar f_*:C_n(\Wi X)\to
C_n(\Wi Y)$ and a cochain homomorphism $\bar f^*:Hom_{ce}(C_n(\Wi
Y),L)\to Hom_{ce}(C_n(\Wi X),L)$. The latter defines a homomorphisms
of the coarsely equivariant cohomology groups
$$\bar f^*_{ce}:H^*_{ce}(\Wi Y;L)\to H^*_{ce}(\Wi X;L).$$

Suppose that $\pi$ acts freely on CW complexes $\Wi X$ and $\Wi Y$
such that the actions preserve the CW complex structures. We call a
cellular map $g:\Wi X\to \Wi Y$ {\em coarsely equivariant} if the set
$$\bigcup_{\gamma\in\pi}\{\gamma^{-1}g_*(\gamma e)\}\subset C_*(\Wi Y)$$
generates a finitely generated subgroup for every cell $e$ in $\Wi
X$ where $g_*:C_*(\Wi X)\to C_*(\Wi Y)$ is the induced chain map. In
other words $g$ is coarsely equivariant if for every cell $e$ the
union $\cup_{\gamma\in\pi}\{\gamma^{-1}g_*(\gamma e)\}$ is contained
in a finite subcomplex.

On a locally finite simplicial complex we consider the geodesic
metric in which every simplex is isometric to the standard.
\begin{prop}\label{induced1}
Let $f:X\to Y$ be a proper coarsely equivariant cellular map. Then the
induced homomorphism on cochains takes the coarsely equivariant
cochains to coarsely equivariant.
\end{prop}
\begin{proof}
Let $\phi:C_n(Y)\to L$ be a coarsely equivariant cochain. Let $e'$ be
a $n$-cell in $X$. There are finitely many cells $c_1,\dots,c_m\in
\Wi Y$ such that $f(\gamma e')\subset \bigcup\gamma\{c_1,\dots, c_m\}$
for all $\gamma\in\pi$. Then the group
$$\langle\bigcup_{\gamma\in\pi}\{\gamma^{-1}\phi(f(\gamma e'))\}\rangle\subset
\langle\bigcup_{\gamma\in\pi}\{\gamma^{-1}\phi(\gamma\{c_1,\dots, c_m\})\}\rangle\subset L
$$ is finitely generated since the later group is generated by finitely
many finitely generated subgroups.
\end{proof}

For every CW-complex $X$ we consider the product CW-complex
structure on $X\times[0,1]$ with the standard cellular structure on
$[0,1]$.

Proposition~\ref{induced1} and the standard facts about cellular
chain complexes imply the following.
\begin{prop}\label{induced2}
Let $X$ and $Y$ be  complexes with free cellular actions of a group
$\pi$.

(A) Then every coarsely equivariant cellular map $f:X\to Y$ induces an
homomorphism of the coarsely equivariant cohomology groups
$$
f^*:H^*_{ce}(Y;L)\to H^*_{ce}(X;L).$$

(B) If two coarsely equivariant maps $f_1, f_2: X\to Y$ are homotopic
by means of a cellular coarsely equivariant homotopy
$H:X\times[0,1]\to Y$, then they induce the same homomorphism of the
coarsely equivariant cohomology groups, $f^*_1=f^*_2$.
\end{prop}

A similar situation is with homologies.
We recall that the equivariant locally
finite homology groups are defined by the complex of infinite {\em
locally finite invariant chains}
$$ C_n^{lf,\pi}(\Wi X;L)=\{\sum_{e\in E_n(\Wi X)} \lambda_ee\mid
\lambda_{ge}=g\lambda_e,\ \lambda_e\in L\}.$$ The local finiteness
condition on a chain requires that for every $x\in\Wi X$ there is a
neighborhood $U$ such that the number of $n$-cells $e$ intersecting $U$
for which $\lambda_e\ne 0$ is finite. This condition is satisfied
automatically when $X$ is a locally finite complex. Even in that
case {\em lf} is the part the notation for the equivariant homology
since it was inherited from the singular theory. The following
proposition implies the equality $H_*(X;L)=H_*^{lf,\pi}(\Wi X;L)$ for finite $X$.
\begin{prop}
For every finite CW complex $X$ with the fundamental group $\pi$ and a
$\pi$-module $L$ the chain complex $\{C_n(\Wi X)\otimes_{\pi}L\}$ is
isomorphic to the chain complex of locally finite equivariant chains
$C_*^{lf,\pi}(\Wi X,L)$.
\end{prop}
For locally finite $X$ one should take direct limits.

Similarly one can define the coarsely equivariant homology groups on a
locally finite CW complex by considering infinite  coarsely equivariant
chains.  Let $X$ be a complex with the fundamental group $\pi$ and
the universal cover $\Wi X$. We call an infinite chain $\sum_{e\in
E_n(\Wi X)} \lambda_ee$ {\em coarsely equivariant} if the group generated by the set
$\{\gamma^{-1}\lambda_{\gamma e}\mid\gamma\in\pi\}\subset L$ is
finitely generated for every cell $e$. As we already have mentioned, the complex
of equivariant locally finite chains defines equivariant locally
finite homology $H^{lf,\pi}_*(\Wi X;L)$. The homology groups defined by the
coarsely equivariant locally finite chains we call  {\em the coarsely
equivariant locally finite homologies}. We denote them as
$H^{lf,ce}_*(\Wi X;L)$.  We note that like in the case of cohomology
this definition can be carried out for the singular homology and it
gives the same groups. In particular the groups $H^{lf,ce}_*(\Wi
X;L)$ do not depend on the choice of a CW complex structure on $X$

As in the case of cohomology for any complex $K$ there is an
equivariant coarsening homomorphism
$$ ec_*^K:H_*(K;L)=H^{lf,\pi}_*(\Wi K;L)\to H_*^{lf,ce}(\Wi K;L)$$
which factors through the perturbation homomorphism
$$pert_*^K:H_*(K;L)\to H^{lf, ae}_*(\Wi K;L)$$ for the almost equivariant homology defined in ~\cite{Dr}.
In the case when $L=\mathbb Z$ is a trivial $\pi$-module the  homomorphism $ec^K_*:H_*(K;\mathbb Z)\to H_*^{lf}(\Wi K;\mathbb Z)$ 
is the transfer generated by the inclusion of invariant infinite chains to locally finite.

There is an analog of Proposition~\ref{induced2} for the
coarsely equivariant locally finite homology.
\begin{prop}\label{h-induced}
Let $X$ and $Y$ be  complexes with free cellular actions of a group
$\pi$.

(A) Then every coarsely equivariant cellular map $f:X\to Y$ induces
a homomorphism of the coarsely equivariant homology groups
$$
f_*:H_*^{lf,ce}(X;L)\to H_*^{lf,ce}(Y;L).$$

(B) If two coarsely equivariant maps $f_1, f_2: X\to Y$ are
homotopic by means of a cellular coarsely equivariant homotopy, then
they induce the same homomorphism of the coarsely equivariant
cohomology groups, $(f_1)_*=(f_2)_*$.
\end{prop}

Let $M$ be a closed oriented $n$-dimensional PL manifold with a fixed
triangulation. Denote by $M^*$ the dual complex. There is a
bijection between $k$-simplices $e$ and the dual $(n-k)$-cells $e^*$
which defines the Poincare duality isomorphism. This bijection
extends to a similar bijection on the universal cover $\Wi M$. Let
$\pi=\pi_1(M)$. For any $\pi$-module $L$ the Poincare duality on $M$
with coefficients in $L$ is given on the cochain-chain level by
isomorphisms
$$
Hom_{\pi}(C_k(\Wi M^*),L) \stackrel{PD_k}\longrightarrow
C_{n-k}^{lf,\pi}(\Wi M;L)$$ where $PD_k$ takes a cochain
$\phi:C_k(\Wi M^*)\to L$ to the following chain $\sum_{e\in
E_{n-k}(\Wi M)}\phi(e^*)e$. The family $PD_*$ is a chain isomorphism
which is also known as the cap product
$$
PD_k(\phi)=\phi\cap F$$ with the fundamental cycle $F\in
C_n^{lf,\pi}(\Wi M)$, where $F=\sum_{e\in E_n(\Wi M)} e$. We
note that the homomorphisms $PD_k$ and $PD_k^{-1}$ extend to the
coarsely equivariant chains and cochains:
$$
Hom_{ce}(C_k(\Wi M^*),L) \stackrel{PD_k}\longrightarrow
C_{n-k}^{lf,ce}(\Wi M;L).$$ Thus, the homomorphisms $PD_*$ define
the Poincare duality isomorphisms $PD_{ce}$ between the coarsely
equivariant cohomology and homology. We summarize this in the
following
\begin{prop}\label{PD}
For any closed oriented $n$-manifold $M$ and any $\pi_1(M)$-module
$L$ the Poincare duality forms the following commutative diagram:
$$
\begin{CD}
H^k(M;L) @ >ec^*_M>> H^k_{ce}(\Wi M;L)\\
@V{\cap [M]}VV @V{PD_{ce}}VV\\
H_{n-k}(M;L) @>ec_*^M>> H_{n-k}^{lf,ce}(\Wi M;L).\\
\end{CD}
$$
\end{prop}

We note that the operation of the cap product for equivariant
homology and cohomology automatically extends on the chain-cochain level
to the cap product on the coarsely equivariant homology and
cohomology. Then the Poincare Duality isomorphism $PD_{ce}$ for $\Wi
M$ can be described as the cap product with the homology class
$ec_*^M([M])=[\Wi M]=[F]$.

\section{Obstruction to the inequality $\dim_{mc}\Wi X<n$}

We consider locally finite CW complexes supplied with a geodesic
metric with finitely many isometry types of $n$-cells for each
$n$. A typical example is a uniform simplicial complex, i.e., a simplicial complex supplied with the geodesic metric 
such that every simplex is isometric to the standard simplex. 
 Let $\pi$ be a finitely presented group. Then the classifying
space $B\pi=K(\pi,1)$ can be taken to be a locally finite simplicial
complex. We fix a geodesic metric on $B\pi$. Let $p_{\pi}:E\pi\to
B\pi$ denote the universal covering. We consider the induced CW
complex structure and induced geodesic metric on $E\pi$.

We recall some notions from the coarse geometry~\cite{Roe}.
A subset $X'\subset X$ of metric space is called {\em coarsely dense}
if there is $D>0$ such that every $D$-ball $B(x,D)$ in $X$ has nonempty intersection with $X'$.
A (not necessarily continuous) map $f:X\to Y$ between metric spaces is called a {\em coarse imbedding}~\cite{Roe}
if there are two nondecreasing tending to infinity functions $\rho_1,\rho_2:\mathbb R_+\to\mathbb R_+$ such that
$$
\rho_1(d_X(x,x'))\le d_Y(f(x),f(x'))\le \rho_2(d_X(x,x'))
$$
for all $x,x'\in X$. 

If $f:K\to B\pi$ is a classifying map of the universal covering $\Wi K$ of finite complex $K$, then any lift $\Wi f:\Wi K\to E\pi$ is a coarse imbedding.
Note that for a geodesic metric space $X$ every coarse imbedding is coarsely Lipschitz.
Also note that every coarse imbedding $f:X\to Y$ is uniformly cobounded. Indeed, for each $a\in\mathbb R_+$ with $\rho_1(a)>0$ we have $diam(f^{-1}(y))\le a$ for all $y$.

A coarse imbedding $f:X\to Y$ is called a {\em coarse equivalence} if  the image $f(X)$ is  coarsely dense in $Y$. Thus for metric CW complexes $X$
that we deal with the inclusion $X^{(n-1)}\subset X^{(n)}$ is a coarse equivalence.
For every coarse equivalence $f:X\to Y$ there is an {\em inverse} coarse equivalence $g:Y\to X$, i.e., a map $g$ such that $f\circ g$ and $g\circ f$ are in bounded distance from the identities
$1_Y$ and $1_X$ respectively.

Let $A$ be a subset of a CW complex $X$. The star neighborhood
$St(A)$ of $A$ is the closure of the union of all cells in $X$ that
have a nonempty intersection with $A$. 
\begin{prop}\label{3.3} Let $X$ and $Y$ be universal covers of locally
finite complexes with fundamental group $\pi$ and let $X'\subset X$
be a coarsely dense $\pi$-invariant subset. Let $\Phi:X\to Y$ be a
coarsely Lipschitz cellular map with the coarsely equivariant
restriction $\Phi|_{X'}$. Then $\Phi$ is coarsely equivariant.
\end{prop}
\begin{proof} The coarsely dense condition implies that there is a number $D>0$
such that for every cell $b\subset X$ there is a cell $e\subset X'$
with $dist(e,b)<D$. Since $\Phi|_{X'}$ is coarsely equivariant, for
every closed cell $e\subset X'$ the union
$$\bigcup_{\gamma\in\pi}\gamma^{-1}\Phi(\gamma\bar e)\subset
\bar\sigma_1\cup\dots\cup\bar\sigma_k$$ lies in the finite union of
closed cells. Then
$$\bigcup_{\gamma\in\pi}\gamma^{-1}\Phi(\gamma\bar b)\subset
St^m(\bar\sigma_1\cup\dots\cup\bar\sigma_k).$$  The existence of $m$
follows from the fact that $\Phi$ is coarsely Lipschitz and
existence of universal $D$. The local finiteness of $Y$ implies that
the $m$-times iterated star neighborhood
$St^m(\bar\sigma_1\cup\dots\cup\bar\sigma_k)$ lies in a finite subcomplex
of $Y$.
\end{proof}
\begin{cor}\label{uniform=almost} Let $X$ and $Y$ be universal covers of
finite complexes with fundamental group $\pi$.  Then a cellular
coarsely Lipschitz homotopy $\Phi:X\times[0,1]\to Y$ of a coarsely
equivariant map is coarsely equivariant.
\end{cor}

\begin{prop}\label{coarse imb} Let $X$ and $Y$ be metric CW complexes.
A coarsely Lipschitz extension $\bar f:X^{(n)}\to Y$ of
a coarse imbedding $f:X^{(n-1)}\to Y$ is a coarse imbedding. In particular, $\bar f$ is uniformly cobounded.
\end{prop}
\begin{proof} 
As we noted above, the inclusion $i:X^{(n-1)}\to X^{(n)}$ is a coarse equivalence. Let $j:X^{(n)}\to X^{(n-1)}$ be a coarse inverse
with $dist(j, 1_{X^{(n)}})<D$.
We show that $\bar f$ is in finite distance from the coarse imbedding $f\circ j$. This would imply that $\bar f$ is a coarse imbedding.
Indeed, 
$$
d_Y(\bar f(x),fj(x))=d_Y(\bar f(x),\bar f(j(x)))\le\lambda d_X(x,j(x))+c\le\lambda D+c.
$$
where $\lambda$ and $c$ are from the definition of the coarsly Lipschitz map $\bar f$.
\end{proof}

\

Here we recall some basic facts of the elementary obstruction
theory. Let $f:X\to Y$ be a cellular map that induces an isomorphism
of the fundamental groups. We want to deform the map $f$ to a map to
the $(n-1)$-skeleton $Y^{(n-1)}$. For that we consider the extension
problem $$X\supset X^{(n-1)} \stackrel{f}{\to} Y^{(n-1)},$$ i.e.,
the problem to extend $f:X^{(n-1)}\to Y^{(n-1)}$ continuously to a
map $\bar f:X\to Y^{(n-1)}$. The primary obstruction for this
problem $o_f$ is  defined by the cochain $C_f: C_n(X)\to\pi_{n-1}(Y^{(n-1)})$ which 
is defined as $C_f(e)=[f\circ\phi_e]$ where $\phi_e:\partial D^n\to X$ is the attaching map of the $n$-cell $e^n$.
It turns out that $C_f$ is a cocycle which defines the obstruction $o_f\in H^n(X;L)$  to extend $f$ to the $n$-skeleton
where
$L=\pi_{n-1}(Y^{(n-1)})$ is the $(n-1)$-dimensional homotopy group
considered as a $\pi$-module for $\pi=\pi_1(Y)=\pi_1(X)$. The
obstruction theory says that a map $g:X\to Y^{(n-1)}$ that agrees
with $f$ on the $(n-2)$-skeleton $X^{(n-2)}$ exists if and only if
$o_f=0$. The primary obstruction is natural: If $g:Z\to X$ is a
cellular map, then $o_{gf}=g^*(o_f)$. In particular, in our case
$o_f=f^*(o_1)$ where $o_1\in H^n(Y;L)$ is the primary obstruction to
the retraction of $Y$ to the $(n-1)$-skeleton.

\begin{defin}
Let $g:Y^{(n-1)}\to Z$ be a coarsely Lipschitz map of the
$(n-1)$-skeleton of an $n$-dimensional  complex to a metric space.
We call the problem to extend $g$ to a coarsely Lipschitz map $\bar
g:Y\to Z$ a {\em coarsely  Lipschitz extension problem }.
\end{defin}

\begin{defin}\label{obst}
Let $X$ be a finite $n$-complex, $n\ge 3$, with $\pi_1(X)=\pi$
and $\Wi f:\Wi X\to \Wi Y$ be a lift of a cellular map $f:X\to Y$
that induces an isomorphism of the fundamental groups. 
We note that the obstruction cocycle
$$C_{\Wi f}:C_n(\Wi X)\to\pi_{n-1}(\Wi Y^{(n-1)})=\pi_{n-1}(Y^{(n-1)})$$
being equivariant is coarsely equivariant. Thus, it defines an
element $o_{\Wi f}\in H^n_{ce}(\Wi X;\pi_{n-1}(Y^{(n-1)})).$  Since it also defines an
element $\kappa_{f}\in H^n_{\pi}(\Wi X;L)=H^n(X;L)$ of the equivariant
cohomology with the $\pi$-module $L=\pi_{n-1}(Y)$, we have $o_{\Wi
f}=ec^*_X(\kappa_{f})$.
\end{defin}

We consider metric CW complex
with the induced metric on their universal coverings.
\begin{prop}\label{obstructiontheory}
Let $\Wi f:\Wi X\to\Wi Y$ be a lift of a Lipschitz cellular map
$f:X\to Y$ of a finite  complex to a locally finite that induces an
isomorphism of the fundamental groups.  Then the above cohomology
class $o_{\Wi f}\in H^n_{ce}(\Wi X;\pi_{n-1}(Y^{(n-1)}))$ is the
primary obstruction for the following coarsely Lipschitz extension
problem
$$\Wi X\supset \Wi X^{(n-1)}\stackrel{\Wi f|}{\to} \Wi Y^{(n-1)}.$$
Thus, $o_{\Wi f}=0$ if and only if there is a coarsely Lipschitz map
$\bar g:\Wi X^{(n)}\to \Wi Y^{(n-1)}$ which agrees with $\Wi f$ on
$\Wi X^{(n-2)}$.
\end{prop}
\begin{proof}
The proof goes along the lines of a similar statement from the
classical obstruction theory. Let $C_{\Wi f}=\delta\Psi$ where
$\Psi:C_{n-1}(\Wi X)\to \pi_{n-1}(\Wi Y^{(n-1)})$ is a coarsely
equivariant homomorphism. For each $(n-1)$-cell $e$ of $X$ we fix a
section $\Wi e\subset\Wi X$, an $(n-1)$-cell in $\Wi X$. Then the
set $\{\gamma^{-1}\Psi(\gamma\Wi e)\mid
\gamma\in\pi\}$ spans a
finitely generated subgroup in $\pi_{n-1}(\Wi Y^{(n-1)})$. Thus this
subgroup is contained in the image of $\pi_{n-1}(F_{\Wi e})$ for
some finite subcomplex $F_{\Wi e}\subset\Wi Y^{(n-1)}$.

Like in the classical obstruction theory we define a map
$g_{\gamma}:\gamma\Wi e\to \gamma F_{\Wi e}\subset Y^{(n-1)}$,
$\gamma\in\pi$ on cells $\gamma\Wi e$ such that $g_{\gamma}$ agrees with
$\Wi f$ outside a small $(n-1)$-ball $B_{\gamma}\subset\gamma\Wi e$ and
the difference of the restriction of $\Wi f$ and $g_{\gamma}$  to $B_{\gamma}$
defines a map $$d_{\Wi f,g_{\gamma}}:S^{n-1}=B_{\gamma}^+\cup B_{\gamma}^-\to
\gamma F_{\Wi e}\subset Y^{(n-1)}$$ that represents the class
$-\Psi(\gamma\Wi e)$. 

The union of $g_{\gamma}$ defines a bounded map
$g:\Wi X^{(n-1)}\to\Wi Y^{(n-1)}$ in such a way that the difference
map $d_{\Wi f,g}:S^{n-1}\to Y^{(n-1)}$ on the cell $\gamma\Wi e$,
$\gamma\in\pi$, represents the element
$-\Psi(\gamma\Wi e)$. Then
the elementary obstruction theory implies that for every $n$-cell
$\sigma'\subset\Wi X$ there is an extension $\bar
g_{\sigma'}:\overline{\sigma'}\to\Wi Y^{(n-1)}$ of
$g|_{\partial\sigma'}$. Since $\partial\sigma'$ is contained in finitely many $(n-1)$-cells, there are $e_1,\dots, e_k\subset X^{(n-1)}$
and $\gamma_1,\dots,\gamma_k\in\pi$ such that $$g(\partial\sigma')\subset Z=\bigcup_{i=1}^k\gamma_iF_{\Wi e_i}.$$
Since $\pi_{n-1}(Z)$ is finitely generated, there is a finite subcomplex $W\subset \Wi Y^{(n-1)}$ containing $Z$ such that $$ker\{\pi_{n-1}(Z)\to\pi_{n-1}(W)\}=ker\{\pi_{n-1}(Z)\to\pi_{n-1}(\Wi Y^{(n-1)})\}.$$
Thus, we may assume that $\bar g_{\sigma'}(\sigma')\subset W$.
If $\gamma\sigma'$ is a translate of $\sigma'$, we obtain $g(\partial\gamma\sigma')\subset\gamma Z$. Therefore we may assume that
$\bar g_{\gamma\sigma'}(\gamma\sigma')\subset \gamma W$.

Thus, the resulting extension $\bar
g:\Wi X\to\Wi Y^{(n-1)}$ is uniformly bounded and hence coarsely
Lipschitz.

In the other direction, if there is a coarsely Lipschitz map $\bar
g:\Wi X\to Y^{(n-1)}$ that coincides with $\Wi f$ on the
$(n-2)$-dimensional skeleton, then the difference cochain $d_{\Wi
f,\bar g}$ is coarsely equivariant. Indeed, for any $\lambda,b>0$
there are only homotopy classes from a finitely generated subgroup
of $\pi_{n-1}(Y)$ can be realized by  coarsely Lipschitz maps
preserving a base point. Then the formula $\delta d_{\Wi f,\bar g}=
C_{\bar g}-C_{\Wi f}$ and the fact that $o_{\bar g}=0$ imply that
$o_{\Wi f}=0$.
\end{proof}

Let $[e]\in \pi_{n-1}(\Wi Y^{(n-1)})$ denote the element of the
homotopy group defined by the attaching map of an $n$-cell $e$. Then
the homomorphism $C_{\Wi 1}: C_n(\Wi Y)\to\pi_{n-1}(\Wi Y^{(n-1)})$
defined as $C_{\Wi 1}(e)=[e]$ is an equivariant cocycle with the
cohomology class $o_{\Wi 1}\in H^n_{ce}(\Wi
Y;\pi_{n-1}(Y^{(n-1)}))$.

\begin{prop}\label{obstruction}
(1) The cohomology class $o_{\Wi f}$ from the above Proposition is
the image under $\Wi f^*$ of the class $o_{\Wi 1}\in H^n_{ce}(\Wi
Y;\pi_{n-1}(Y^{(n-1)}))$.

(2) The class $o_{\Wi 1}$ comes under the equivariant coarsening homomorphism
$ec_{\pi}^*$ from the primary obstruction $\kappa_1\in H^n(Y;
\pi_{n-1}(Y^{(n-1)}))$ to retract $Y$ to the $(n-1)$-dimensional
skeleton.
\end{prop}
\begin{proof}
The first part is the naturality of obstructions for coarsely
Lipschitz extension problems with respect to coarsely Lipschitz
maps. Like in the case of classical obstruction theory, it follows
from the definition.

The second part follows from definition (see Definition~\ref{obst}).
\end{proof}

\begin{thm}\label{obstr-dim}
Let $X$ be a finite $n$-complex with $\pi_1(X)=\pi$ and let $f:X\to
B\pi$ be a  map that induces an isomorphism of the fundamental
groups where $B\pi$ is a locally finite CW complex. Then $\dim_{mc}\Wi X<n$ if and only if the above obstruction
is trivial, $o_{\Wi f}=0$.
\end{thm}
\begin{proof}
If $\dim_{mc}\Wi X<n$, then by Theorem~\ref{st-cech} there is a
coarsely Lipschitz cellular homotopy of $\Wi f:\Wi X\to E\pi$ to a
map $g:\Wi X\to E\pi^{(n-1)}$. By Corollary~\ref{uniform=almost},
the map $g$ is coarsely equivariant. Then by
Proposition~\ref{induced2}, $o_{\Wi f}=\Wi f^*(o_1)=g^*i^*(o_1)=0$
where $i:E\pi^{(n-1)}\subset E\pi$ is the inclusion.

We assume that $B\pi$ is a locally finite simplicial complex. If
$o_{\Wi f}=0$, then by Proposition~\ref{obstructiontheory} there is
a coarsely Lipschitz continuous map $g:\Wi X\to E\pi^{(n-1)}$ which agrees with
$\Wi f$ on the $(n-2)$-skeleton $X^{(n-2)}$. 
Note that the restriction $\Wi f|_{X^{(n-2)}}$ is a coarse imbedding.
By  consecutive applications of
Proposition~\ref{coarse imb} we obtain that $g$ is uniformly cobounded.
Therefore, $\dim_{mc}\Wi X<n$.
\end{proof}

\section{Obstruction for a manifold to be $md$-small}

We recall that the Berstein-Schwarz class $\beta\in H^1(\pi,I(\pi))$ of a group $\pi$ is define by the cochain
on the Cayley graph $\phi: G\to I(\pi)$ which takes an ordered  edge $[g,g']$ between the vertices labeled by $g,g'\in G$ to $g'-g$. Here $I(\pi)$ is the augmentation ideal of
the group ring $\mathbb Z\pi$. We use notation $I(\pi)^k$ for the $k$-times tensor product $I(\pi)\otimes\dots\otimes I(\pi)$
over $\mathbb Z$. Then the cup product $\beta\smile\dots\smile\beta$ is defined as an element of $H^k(\pi,I(\pi)^k)$.

The Berstein-Schwarz class is universal in the following sense.
\begin{thm}[\cite{Sw}, \cite{Be}, \cite{DR}]\label{universal}
For every $\pi$-module $L$ and every element $\alpha\in H^k(\pi,L)$ there is a $\pi$-homomorphism $\xi:I(\pi)^k\to L$ such that $\xi^*(\beta^k)=\alpha$ where $$\xi^*:H^k(\pi,I(\pi)^k)\to H^k(\pi,L)$$ is the coefficient homomorphism.
\end{thm}

\begin{cor}
The cohomological dimension of a group $\pi$ can be computed by the formula
$$
cd(\pi)=\max\{n\mid \beta^n\ne 0\}
$$
where $\beta$ is the Berstein-Schwarz class of $\pi$.
\end{cor}

The following is well-known (see \cite{BD}, Proposition 3.2.).
\begin{thm}\label{iness}
For a closed oriented $n$-manifold $M$ with the classifying map $f:M\to B\pi$ the following are equivalent:

1. $M$ is inessential;

2. $f_*([M])=0$ in $H_*(B\pi;\mathbb Z)$;

3. $f^*(\beta^n)=0$ in $H^*(M;I(\pi)^n)$ where $\beta$ is the Berstein-Schwarz class of $\pi$.
\end{thm}

The following can be considered as a coarse analog of Theorem~\ref{iness}.

\begin{thm}\label{small}
For a closed oriented $n$-manifold $M$ with the classifying map $f:M\to B\pi$ and its lift to the universal covers
$\Wi f:\Wi M\to E\pi$ the following are equivalent:

1. $M$ is md-small;

2. $\Wi f_*([\Wi M])=0$ in $H^{lf}_n(E\pi;\mathbb Z)$ where $[\Wi M]\in H^{lf}_n(\Wi M;\mathbb Z)$ is the fundamental class of $\Wi M$;

3. $f_*([M])\in ker(ec_*^{\pi})$ where $[M]$ is the fundamental class of $M$;

4. $f^*(\beta^n)\in ker(ec^*_M)$ where $\beta$ is the Berstein-Schwarz class of $\pi$.
\end{thm}
\begin{proof}
1. $\Rightarrow$ 2. We may assume that $f:M\to B\pi$ is cellular  and Lipschitz 
for some metric CW complex structure on $B\pi$. If $\dim_{mc}\Wi M<n$, then by
Corollary~\ref{2.4} there is a coarsely Lipschitz cellular
homotopy of $\Wi f:\Wi X\to E\pi$ to a map $g:\Wi X\to E\pi^{(n-1)}$
with a compact projection to $B\pi$. By
Proposition~\ref{uniform=almost}, it is coarsely equivariant. Then
by Proposition~\ref{induced2} it follows that $\Wi
f_*([\Wi M]))=0$.

2. $\Rightarrow $ 3. Note that $$ec_*^{\pi}(f_*([M]))=\Wi
f_*(ec_*^M([M]))=\Wi f_*([\Wi M])=0$$ and
hence, $f_*([M])\in ker(ec_*^{\pi})$.

3. $\Rightarrow $ 4.  If $f_*([M])\in ker(ec_*^{\pi})$, then
$ec_*^{\pi}(f_*([M])\cap\beta^n)=0$. Since  the commutative
diagram
$$
\begin{CD}
H_0^{lf,ce}(\Wi M;I(\pi)^n) @>\bar f_*>> H_0^{lf,ce}(E\pi;I(\pi)^n)\\
@Aec_*^MAA @Aec_*^{\pi}AA\\
H_0(M;I(\pi)^n) @>f_*>> H_0(B\pi;I(\pi)^n)\\
\end{CD}
$$
has isomorphisms for horizontal arrows, $ec_*^M([M]\cap
f^*(\beta^n))=0$. Thus, $ec_*^M([M])\cap ec^*_Mf^*(\beta^n)=0$.
By the Poincare Duality, $ec^*_Mf^*(\beta^n)=0$.

4. $\Rightarrow$ 1. We show that the obstruction $o_{\Wi f}$ to the inequality $\dim_{mc}\Wi M<n$ is  zero
and apply Theorem~\ref{obstr-dim}. By Theorem~\ref{universal} there is a $\pi$-homomorphism $\xi:I(\pi)^n\to L=\pi_{n-1}(B\pi^{(n-1)})$ such that $\xi^*(\beta^n)=\kappa_1$ and $\xi^*(f^*(\beta^n))=o_f$ where $\kappa_1$ is the primary obstruction to retract $B\pi$ onto $B\pi^{(n-1)}$ and $o_f$ is the primary obstruction do deform $f$ into $B\pi^{(n-1)}$. Then
$o_{\Wi f}=ec_M^*(o_f)=\xi^*ec_M^*(\beta^n)=0$.
\end{proof}
\begin{cor}\label{small homology}
For every finitely presented group $\pi$ and every $n$ there is a subgroup $H^{sm}_n(B\pi)\subset H_n(B\pi;\mathbb Z)$ of $md$-small classes such that

(1) If there is an $md$-small orientable manifold $M$ with a classifying map $f:M\to B\pi$, then $f_*([M])\in H^{sm}_*(B\pi)$;

(2) If a class $\alpha\in H^{sm}_*(B\pi)$ is the image $\alpha=f_*([M])$ of the fundamental class of a manifold $M$ for a classifying map $f:M\to B\pi$, then $M$ is $md$-small.
\end{cor}
\begin{proof}
We define $H^{sm}_*(B\pi)=ker (ec^{\pi}_*)$.
\end{proof}
For the class of MD-large manifolds similar theorem was proven in~\cite{Dr}.

\

Corollary~\ref{small homology} states that the property for manifolds to be $md$-small is a group homology property. This result is in the spirit of the results of Brunnbauer and Hanke~\cite{BH} where the notion of the small homology subgroup $H_*^{sm}(\pi)\subset H_*(\pi)$
for a given class of large manifolds was first introduced. They proved similar results for several classes of large manifolds. 
The major difference is that all their results are rational and deal with the rational homology $H_*(\pi;\mathbb Q)$.
It is an open question if Corollary~\ref{small homology} holds true rationally. In particular, it is unknown
whether all torsion classes in $H_*(\pi)$ are small. A related question is
\begin{question}
Suppose that the connected sum $M\# M$ of a manifold $M$ with itself is $md$-small. Does it follow that $M$ is $md$-small?
\end{question}

\begin{thm}
$M$ is md-small if and only if $M\times S^1$ is md-small.
\end{thm}
\begin{proof} Let $n=\dim M$ and let $\Wi f:\Wi M\to E\pi$ denote a lift to the universal covers of a classifying map $f:M\to B\pi$.
In the following commutative diagram
$$
\begin{CD}
H^{lf}_n(\Wi M;\mathbb Z) @>{\Wi f_*}>> H^{lf}_n(E\pi;\mathbb Z)\\
@VVV @VVV\\
H^{lf}_{n+1}(\Wi M\times\mathbb R;\mathbb Z) @>{(\Wi f\times 1)_*}>> H^{lf}_{n+1}(E\pi\times\mathbb R;\mathbb Z)\\
\end{CD}
$$
the vertical arrows are the suspension isomorphisms. Then $\Wi f_*=0$ if and only if $(\Wi f\times 1)_*=0$  and the result follows from Theorem~\ref{small} and the fact that $ec^{\pi}_*f_*([M])=\Wi f_*ec^M_*([M])$.
\end{proof}

\section{MD-large manifolds that are $md$-small}

A similar to Theorem~\ref{small} criterion for manifolds to be MD-small  was proven in~\cite{Dr}:
\begin{thm}\label{Small}
For a closed oriented $n$-manifold $M$ with a classifying the universal cover map $f:M\to B\pi$ the following are equivalent:

1. $M$ is MD-small;

2. $f_*([M])\in ker(pert_*^{\pi})$;

3. $f^*(\beta^n)\in ker(pert^*_M)$ where $\beta$ is the Berstein-Schwarz class.
\end{thm}
A group $\pi$ is called {\em geometrically finite} if there is a finite Eilenberg-Mclane complex $K(\pi,1)$.
We recall that a group $\pi$ is called a {\em duality group}~\cite{Br} 
if there is a module $D$ such that
$$
H^i(\pi,M)\cong H_{m-i}(\pi,M\otimes D)
$$
for all $\pi$-modules $M$ and all $i$ where $m=cd(\pi)$.

\begin{thm}\label{duality}
Let $\pi$ be a geometrically finite duality group. Then $H^{lf}_i(E\pi;\mathbb Z)=0$ for all $i\ne cd(\pi)$.
\end{thm}
\begin{proof} 
From Theorem 10.1, Chapter 8 of~\cite{Br} it follows that $H^i(\pi,\mathbb Z\pi)=0$ for $i\ne m= cd(\pi)$ and
$H^m(\pi,\mathbb Z\pi)$ is a free abelian group. In view of the equality $H^i(\pi,\mathbb Z\pi)=H_c^i(E\pi;\mathbb Z)$
for geometrically finite groups
(see~\cite{Br} Theorem 7.5, Chapter 8) and the short exact sequence for the Steenrod homology of a compact metric space
$$
0\to Ext(H^{i+1}(X),\mathbb Z)\to H^s_i(X;\mathbb Z)\to Hom(H^i(X),\mathbb Z)\to 0
$$
applied to the one point compactification $\alpha(E\pi)$ of $E\pi$ we obtain that $H^s_i(\alpha(E\pi);\mathbb Z)=0$ for
$i< m$. The equality $H^{lf}_i(E\pi;\mathbb Z)=H^s_i(\alpha(E\pi);\mathbb Z)$ completes the proof.
\end{proof}

\begin{thm}\label{ex}
Every rationally essential $n$-manifold  $M$ whose fundamental group $\pi$ is a geometrically finite amenable duality group with $cd(\pi)>n$  is  md-small and MD-large.
\end{thm}
\begin{proof}
It was proven in~\cite{Dr2}, Theorem 7.6, that a rationally essential manifold with an amenable fundamental group is
MD-large. Here we present a sketch of an alternative proof of that theorem for geometrically finite groups suggested to me by M. Marcinkowski. 
He  noticed that
the perturbation homomorphism $pert^{\pi}_*:H_*(\pi,\mathbb Z)\to H_*^{lf,ae}(E\pi;\mathbb Z)$ can be identified with the coefficient homomorphism $$i_*:H_*(\pi,\mathbb Z)\to H_*(\pi,\ell^{\infty}(\pi,\mathbb Z))$$ induced by the inclusion $i:\mathbb Z\to \ell^{\infty}(\pi,\mathbb Z)$ of  the constant functions into the $\pi$-module of all bounded functions. This homomorphism $i_*$ and the similar homomorphism 
$\bar i_*:H_*(\pi,\mathbb R)\to H_*(\pi,\ell^{\infty}(\pi,\mathbb R))$ for coefficients in $\mathbb R$ form a commutative diagram
$$
\begin{CD}
H_*(\pi,\mathbb Z) @>i_*>> H_*(\pi,\ell^{\infty}(\pi,\mathbb Z))\\
@VVV @VVV\\
H_*(\pi,\mathbb R) @>\bar i_*>> H_*(\pi,\ell^{\infty}(\pi,\mathbb R)).\\
\end{CD}
$$
In view of the amenability of $\pi$, the homomorphism $\bar i:\mathbb R\to \ell^{\infty}(\pi,\mathbb R)$ is a split injection. Therefore, $\bar i_*$ is
an injection. Since $M$ is rationally essential, $f_*([M])\ne 0$ in $H_*(\pi,\mathbb R)$. Therefore, $pert^{\pi}_*(f_*[M])\ne 0$ and hence $M$ is MD-large by Theorem~\ref{Small}.

Since $\pi$ is a duality group, we obtain $H^{lf}_n(E\pi;\mathbb Z)=0$ for $n< cd(\pi)$.
The coarsely equivariant homology group $H^{lf,ce}_*(E\pi;\mathbb Z)$ coincides with $H^{lf}_*(E\pi;\mathbb Z)$.
By Theorem~\ref{duality}, $H_n^{lf}(E\pi;\mathbb Z)=0$, therefore,  by Theorem~\ref{small}, $M$ is $md$-small.
\end{proof}

REMARK. In the definition of $\dim_{MC}$ in~\cite{Dr2}, Definition 1.1, the Lipschitz condition on $f$ was omitted by a mistake. It was used implicitly in the proof of Theorem 7.6.
\\
\\

EXAMPLE. For $\pi=\mathbb Z^m$, $m>5$, every homology class $\alpha\in H_n(B\pi;\mathbb Q)=\oplus\mathbb Q$, $n<m$  can be realized by a manifold $f:M\to B\pi$ (due to Thom's theorem). For $n>4$ applying surgery in dimension 0 and 1 we can achieve that $f$ induces isomorphism of the fundamental groups. By Theorem~\ref{ex}, $M$ is $md$-small and MD-large simultaneously.

\end{document}